\documentclass[reqno, oneside, 10pt]{article}

\usepackage{amsmath,amsthm,amssymb}
\usepackage{bbm}
\usepackage{geometry}

\newtheorem{theorem}{Theorem}
\newtheorem{lemma}[theorem]{Lemma}
\newtheorem{corollary}[theorem]{Corollary}
\newtheorem{remark}[theorem]{Remark}

\newtheorem{conj}[theorem]{Conjecture}

\newcommand{\cC}{{\mathcal C}}
\newcommand{\cD}{{\mathcal D}}

\newcommand{\cJ}{{\mathcal J}}

\newcommand{\cI}{{\mathcal I}}
\newcommand{\cK}{{\mathcal K}}

\newcommand{\cT}{{\mathcal T}}
\newcommand{\cTp}{{\mathcal T^+}}

\newcommand{\fS}{{\mathfrak S}}

\newcommand{\oM}{{\overline{M}}}

\newcommand\Bi{\mathrm{Bi}}

\newcommand\E{{\mathbb E}}
\newcommand\Var{\operatorname{Var}} 
\renewcommand\Pr{{\mathbb P}}

\newcommand\NN{\mathbb{N}}

\newcommand\floor[1]{\lfloor {#1} \rfloor}

\newcommand\bigceil[1]{\big\lceil {#1} \big\rceil}
\newcommand\ceil[1]{\lceil {#1} \rceil}

\newcommand\Gnp{\ensuremath{G_{n,p}}}

\renewcommand{\epsilon}{\varepsilon}
\newcommand{\eps}{\varepsilon}

\newenvironment{romenumerate}[1][-5pt]{
\addtolength{\leftmargini}{#1}\begin{enumerate}
\itemsep0pt \parskip0pt \parsep0pt%
 }{\end{enumerate}}

\long\def\symbolfootnote[#1]#2{\begingroup
\def\thefootnote{\fnsymbol{footnote}}\footnote[#1]{#2}\endgroup}

\renewcommand{\emptyset}{\varnothing} 

\newcommand{\indic}[1]{\mathbbm{1}_{\{{#1}\}}}

\let\OLDthebibliography\thebibliography
\renewcommand\thebibliography[1]{
  \OLDthebibliography{#1}
  \setlength{\parskip}{0pt}
  \setlength{\itemsep}{0pt plus 0.3ex}
}

\begin{document}

\title{Upper tail bounds for Stars} 
\author{Matas {\v{S}}ileikis\thanks{Department of Theoretical Computer Science, Institute of Computer Science of the Czech Academy of Sciences, 182~07~Prague, Czech Republic. 
E-mail: {\tt matas.sileikis@gmail.com}. With institutional support RVO:67985807. Research supported by the Czech Science Foundation, grant number GJ16-07822Y.} \ 
 and Lutz Warnke\thanks{School of Mathematics, Georgia Institute of Technology, Atlanta GA~30332, USA.
E-mail: {\tt warnke@math.gatech.edu}. 
Research partially supported by NSF Grant DMS-1703516 and a Sloan Research Fellowship. 
Part of the work was done while the author was a member of the Department of Pure Mathematics and Mathematical Statistics, University of Cambridge.
}}
\date{29 January 2019} 
\maketitle

\begin{abstract}
For $r \ge 2$, let~$X$ be the number of $r$-armed stars~$K_{1,r}$ in the binomial random graph~$\Gnp$. 
We study the upper tail $\Pr(X \ge (1+\eps)\E X)$, and establish exponential bounds which are best possible up to constant factors in the exponent 
(for the special case of stars~$K_{1,r}$ this solves a problem of Janson and Ruci{\'n}ski, and confirms a conjecture by DeMarco and Kahn). 
In contrast to the widely accepted standard for the upper tail problem, we do not restrict our attention to constant~$\eps$, but also allow for $\eps \ge n^{-\alpha}$ deviations. 
\end{abstract}

\section{Introduction}
The study of (the distribution of) small subgraphs in the binomial random graph~$\Gnp$ is one of the most fundamental and influential problems in the theory of random graphs.  
Starting with the seminal work of Erd{\H o}s and R\'enyi~\cite{ER1960} from~1960, the early results for the number~$X_H$ of copies of~$H$ in $\Gnp$ concerned the threshold of appearance (i.e., when~$\Pr(X_H>0) \to 1$) and the range of edge-probabilities~$p$ for which~$X_H$ is asymptotically~normal; 
these basic features were eventually resolved in the~1980s by Bollob{\'a}s~\cite{SGTR} and Ruci{\'n}ski~\cite{SGN}. 
Later the focus changed to finer details of the distribution of $X_H$, and the \emph{lower tail} $\Pr(X_H \le (1-\eps)\E X_H)$ was studied intensively in the late 1980s (often for the special case $\eps=1$). 
This led to the discovery of  Janson's inequality~\cite{Janson,JansonJLR,RiordanWarnke2012J}, which gives exponential bounds for $\Pr(X_H \le (1-\eps)\E X_H)$ that are best possible up to constant factors in the exponent (cf.\ the recent work of Janson and Warnke~\cite{JW}).

Since the early~1990s the `infamous' \emph{upper tail} $\Pr(X_H \ge (1+\eps)\E X_H)$ has remained an important challenge, providing a well-known testbed for 
concentration inequalities (see, e.g.,~\cite{UT}). 
After polynomial bounds around~1990 by Spencer~\cite{Spencer1990} and exponential bounds in the late~1990s 
via the Kim--Vu polynomial concentration method~\cite{KimVu2000,Vu2001}, 
in~2002 Janson, Oleszkiewicz and Ruci{\'n}ski~\cite{UTSG} obtained a breakthrough: 
via a moment based method they obtained exponential estimates for~$\Pr(X_H \ge (1+\eps)\E X_H)$ that, for constant~$\eps$, are best possible up to logarithmic factors in the exponent (see also~\cite{UTHG,UTAP} for extensions to random hypergraphs, and arithmetic progressions in random subsets of integers).  
The \emph{upper tail problem} of closing the aforementioned logarithmic gap 
has remained open during the last decade, and only recently this has been settled for 
cliques~$K_r$ by DeMarco and Kahn~\cite{K3TailDK,KkTailDK} (see also Chatterjee~\cite{K3TailCh} for~$r=3$), and for arithmetic progressions by Warnke~\cite{APUT}. 
Modern large deviation theory also gives partial results~\cite{CD2014,LZ2014,BGLZ17,DC2018} for large edge-probabilities~$p \ge n^{-\delta_H}$ (this restriction sidesteps some difficulties of the upper tail problem).

\enlargethispage{\baselineskip}

In this paper we solve the upper tail problem for $r$-armed stars~$K_{1,r}$, 
and as a conceptual novelty we will also allow~$\eps$ to depend on~$n$ (i.e., do not restrict our attention to constant~$\eps$, as usual). 
The casual reader might suspect that tail estimates for $r$-armed stars are essentially trivial, but this is only true for~$r=1$ (where $X_{K_{1,1}} = |E(\Gnp)|$ since $K_{1,1}=K_2$). 
To put this into context, Janson, Oleszkiewicz and Ruci\'nski~\cite{UTSG} proved that for $r$-regular graphs $H$, such as cliques~$K_{r+1}$, the upper tail satisfies 
\begin{equation*}\label{eq:H:JOR}
p^{O_{H,\eps}(n^2p^r)} \le \Pr(X_{H} \ge (1+\eps) \E X_{H}) \le e^{-\Omega_{H,\eps}(n^2p^r)} ,
\end{equation*}
where the subscripts in~$O_{H,\eps}$ and~$\Omega_{H,\eps}$ indicate that the implicit constants may depend on~$H$ and~$\eps$. 
They also highlighted $K_{1,r}$ (with $r \ge 2$) as key example where 
the form of the exponent 
is more complicated, 
i.e, has different expressions for different ranges of $p$. 
This surprising intricacy is further manifested by the history of the infamous upper tail problem. 
Namely, Vu~\cite{Vu2001} argued in 2000 that his general results were 
essentially unimprovable due to $r$-armed stars, 
for which he obtained  bounds of the~form 
\begin{equation}\label{eq:K1r:VU}
p^{O_{r,\eps}(n^{1+1/r}p)} \le \Pr(X_{K_{1,r}} \ge (1+\eps) \E X_{K_{1,r}}) \le e^{-\Omega_{r,\eps}(n^{1+1/r}p)} .
\end{equation}
However, Janson, Oleszkiewicz and Rucinski~\cite{UTSG} later discovered that the upper tail behavior of $K_{1,r}$ is more delicate 
(the lower bound in \eqref{eq:K1r:VU} is not always correct), 
and obtained bounds of the more involved form 
\begin{equation}\label{eq:K1r:JOR}
p^{O_{r,\eps}(\max\{n^{1+1/r}p, \: n^2p^r\})} \le \Pr(X_{K_{1,r}} \ge (1+\eps) \E X_{K_{1,r}}) \le e^{-\Omega_{r,\eps}(\max\{n^{1+1/r}p, \: n^2p^r\})} .
\end{equation}
In words, for stars the form of the upper tail changes around~$p \approx n^{-1/r}$, which is an intriguing phenomenon (that does not occur for cliques). 
In fact, a recent conjecture of DeMarco and Kahn~\cite{KkTailDK,SWUT} for general~$H$ suggests that in~\eqref{eq:K1r:JOR} the `correct' exponent involves yet another term~$\mu:= \E X_{K_{1,r}} = \Theta_r(n^{r+1}p^{r})$, see~\eqref{eq:thm:Krp}. 
However, despite some partial results~\cite{Matas2012,Matas2012phd,LW16}, 
the quest for matching bounds in~\eqref{eq:K1r:VU}--\eqref{eq:K1r:JOR} remained open.

\subsection{Main results} %
Our first basic result settles the upper tail problem of $r$-armed stars for constant~$\eps$, 
by closing the existing~$\log(1/p)$ gap in the exponent for \emph{all}~$p \in (0,1]$. 
In particular, \eqref{eq:thm:Krp} below confirms\footnote{Using Corollary~1.8 in~\cite{UTSG} and the discussion of~Remark 8.3 in~\cite{UTSG} 
it is not difficult to check that the special case~$H=K_{1,r}$ of Conjecture~10.1 in~\cite{KkTailDK} indeed reduces to~\eqref{eq:thm:Krp} with~$\eps=1$; see also equation~(4.27) in~\cite{Matas2012phd} and Remark~2 in~\cite{Matas2012}.}
 Conjecture~10.1 of DeMarco and Kahn~\cite{KkTailDK} in the special case~$H=K_{1,r}$. 
For subgraph counts this is the first example of a sharp upper tail estimate where, for constant~$\eps$, the form of the exponent undergoes \emph{multiple phases} 
(i.e., has more more than two different expressions for different ranges of~$p$). 
%
\begin{theorem}[Upper tail problem for constant~$\eps$]\label{thm:Krp}
Given~$r \geq 2$, let~$X=X_{r,n,p}$ be the number of copies of~$K_{1,r}$ in~$\Gnp$. 
Set~$\mu:=\E X$. 
For~$p\in (0,1]$ and~$\eps>0$ satisfying~$1 \le (1+\eps)\mu \leq X_{r,n,1}$ we have 
\begin{equation}\label{eq:thm:Krp}
	-\log \Pr(X \geq (1+\eps)\mu) = \Theta_{r,\eps}\bigl(\Phi\bigr) \quad \text{ with } \quad \Phi := \min\Bigl\{\mu, \; \max\bigl\{\mu^{1/r},  \mu/n^{r-1}\bigr\} \log (1/p)\Bigr\} .
\end{equation}
\end{theorem}
\noindent
Note that the assumption~$(1+\eps)\mu \leq X_{r,n,1}$ is necessary (since~$X>X_{r,n,1}$ is impossible), 
and that the assumption~$(1+\eps)\mu \ge 1$ is natural (since otherwise~$\Pr(X \geq (1+\eps)\mu) = \Pr(X \ge 1) = 1 - \Pr(X=0)$ holds). 
We now motivate the intricate form of the exponent in~\eqref{eq:thm:Krp} for~$\eps=1$.   
First, Poisson approximation heuristics suggest that~$\Pr(X \ge 2\mu) \approx e^{-\Theta(\mu)}$ for small edge-probabilities~$p$. 
Second, it turns out that~$m=\Theta_r(\max\{\mu^{1/r},\mu/n^{r-1}\})$ appropriately clustered\footnote{For example, complete bipartite graphs~$K_{y,z}$ with suitable~$y = \Theta_r(\min\{\mu^{1/r},n\})$ and~$z=\Theta_r(\mu/y^r)$ suffice: they contain~$\ge z \binom{y}{r} = \Theta_r(z y^r)= \Theta_r(\mu)$ stars and $yz = \Theta_{r}(\mu/y^{r-1}) = \Theta_r(\max\{\mu^{1/r},\mu/n^{r-1}\})$ edges; see Lemma~\ref{lem:cls} for more details.\label{fn:bipartite}}
edges~$F$ suffice to create $2\mu$ copies of~$K_{1,r}$, which implies~$\Pr(X \ge 2\mu) \ge \Pr(F \subseteq G_{n,p}) \ge p^m = e^{m\log (1/p)}$. 
Intuitively, Theorem~\ref{thm:Krp} confirms that the more likely of these two mechanisms (the one with larger probability) 
controls the upper tail behaviour for constant~$\eps$.

\enlargethispage{\baselineskip}

Our second result determines 
the correct dependence of the stars upper tail on~$\eps$, up to constant factors in the exponent 
(this contrasts Theorem~\ref{thm:Krp} above, where the implicit constants may depend on~$\eps$). 
In particular, \eqref{eq:thm:Kri} below solves Problem~6.1 of Janson and Ruci{\'n}ski~\cite{DL} in the special case~$H=K_{1,r}$. 
For subgraph counts this is the first example where, for~$p$ bounded away from one, the order 
of the \emph{large deviation rate function} $-\log \Pr(X \geq (1+\eps)\mu)$ is determined for~$\eps=\eps(n)$ of form~$\eps \ge n^{-\alpha}$   
(the assumption~$\Phi(\eps) \ge 1$ is natural, since it ensures that we are dealing with exponentially small probabilities).   
\begin{theorem}[Upper tail problem for~$\eps =\eps(n) \ge n^{-\alpha}$]\label{thm:Kri}
Given $r \geq 2$, let $X=X_{r,n,p}$ be the number of copies of $K_{1,r}$ in $\Gnp$. 
Set $\mu:=\E X$, $\sigma^2 := \Var X$, and~$\varphi(x):=(1+x)\log(1+x)-x$. 
Given $\xi \in (0,1)$ there is~$\alpha=\alpha(r)>0$  
such~that, for~$p\in(0,1-\xi]$ and~$\eps \ge n^{-\alpha}$ satisfying~$\Phi(\eps) \ge 1$ and~$1 \le (1+\eps)\mu \leq X_{r,n,1}$, we~have 
\begin{equation}\label{eq:thm:Kri}
	-\log \Pr(X \geq (1+\eps)\mu) = \Theta_{r,\xi}\bigl(\Phi(\eps)\bigr) ,	
\end{equation}
with
\begin{equation}\label{eq:thm:Kri:Phi}
\Phi(\eps):=\min\Bigl\{\varphi(\eps)\mu^2/\sigma^2, \; \max\bigl\{(\eps \mu)^{1/r},(\eps \mu)/n^{r-1}\bigr\} \log (e/p)\Bigr\} .
\end{equation}
\end{theorem}
\begin{remark}\label{rem:variance}
The variance satisfies~$\sigma^2 = \Theta_r((1-p)\mu(1+(np)^{r-1}))$; see, e.g., Lemma~3.5 in~\cite{JLR}.  
Furthermore, if~$\mu^{1-1/r} \ge \log n$ holds, then in~\eqref{eq:thm:Kri:Phi} we can replace~$\varphi(\eps)\mu^2/\sigma^2$ by~$(\eps\mu)^2/\sigma^2$; see~Lemma~\ref{cl:phi_to_t_square}. 
\end{remark}
\begin{conj}[Correct upper tail behaviour]\label{conj:Kri}
Theorem~\ref{thm:Kri} remains valid without the assumption~$\eps \ge n^{-\alpha}$.  
\end{conj}
\noindent
We now motivate the somewhat unusual form of the exponent in~\eqref{eq:thm:Kri}. 
First, normal approximation heuristics\footnote{The same normal heuristic suggests that in~\eqref{eq:thm:Krp} we should perhaps have used~$\mu^2/\sigma^2$ instead of~$\mu$, 
but it turns out that then the~$\mu^2/\sigma^2$ term would only matter for~$\Phi$ (i.e., determine the minimum) in a range of~$p$ where~$\mu^2/\sigma^2=\Theta_{r,\eps}(\mu)$ holds.} 
suggest that~$\Pr(X \ge (1+\eps)\mu) \approx e^{-\Theta_r((\eps \mu)^2/\sigma^2)}$ for very small~$\eps$, 
and this sub-Gaussian tail is consistent with the~$\varphi(\eps)\mu^2/\sigma^2$ term in~\eqref{eq:thm:Kri:Phi} since~$\varphi(\eps)=\Theta(\eps^2)$ as~$\eps \to 0$ (the function~$\varphi$ is well-known from Chernoff bounds). 
Second, in~$G_{n,p}$ we usually expect to have at least~$(1-\eps)\mu$ copies of~$K_{1,r}$, say, so enforcing~$2\eps\mu$ extra copies
 via~$m=\Theta_r(\max\{(\eps\mu)^{1/r},(\eps\mu)/n^{r-1}\})$ appropriately clustered\footnote{As before, complete bipartite graphs~$K_{y,z}$ with  suitable~$y = \Theta_r(\min\{(\eps\mu)^{1/r},n\})$ and~$z=\Theta_r(\eps\mu/y^r)$ suffice (see~Lemma~\ref{lem:cls}).} 
edges~$F$ should thus be enough to give a total of~$(1+\eps) \mu$ copies of~$K_{1,r}$; this heuristic 
loosely suggests~$\Pr(X \ge (1+\eps)\mu) \ge \Omega_r(1) \cdot \Pr(F \subseteq G_{n,p}) \ge \Omega_r(p^m) 
\ge e^{-O_r(m\log (1/p))}$. 
Intuitively, Conjecture~\ref{conj:Kri} predicts that the form of the upper tail is indeed 
determined by either 
sub-Gaussian or 
`clustered'~behaviour, 
and Theorem~\ref{thm:Kri} confirms this~for~$\eps =\eps(n) \ge n^{-\alpha}$.

\enlargethispage{\baselineskip}

Our third result approaches the upper tail problem from a conceptually slightly different perspective, studying~$\Pr(X \ge \mu +t)$ for general deviations~$t$ 
(this contrasts Theorem~\ref{thm:Krp} and~\ref{thm:Kri} above, where we focus on the large deviations range~$t=\eps\mu$ and then put restrictions on~$\eps$). 
For subgraph counts, inequality~\eqref{eq:thm:Krip} below is the first example where, for moderately large edge-probabilities~$p$, 
the order of~$-\log \Pr(X \geq \mu+t)$ is \emph{completely resolved} for all exponentially small deviations (where~$t \ge \sigma$ is the natural target assumption). 
We complement this result with inequality~\eqref{eq:thm:Kripc} below, which is the first example where the order of~$-\log \Pr(X \geq \mu+t)$ is resolved for nearly 
all deviations~$t$ where the `clustered' behaviour determines the exponent 
(here~$t^2/\sigma^2 \ge M(t) \log (e/p)$ is the natural target assumption for~$\mu^{1-1/r} \ge \log n$;  
see~\eqref{eq:thm:Kri:Phi}, Remark~\ref{rem:variance}, and Conjecture~\ref{conj:Kri}). 
\begin{theorem}[General upper tail bounds: moderate deviations and clustered regime]\label{thm:Kripc}
Given $r \geq 2$, let $X=X_{r,n,p}$ be the number of copies of $K_{1,r}$ in $\Gnp$. Set~$\mu:=\E X$ and~$\sigma^2 := \Var X$. 
Given $\xi \in (0,1)$, then the following holds whenever~$p \in (0,1-\xi]$ and~$1 \le \mu + t \le X_{r,n,1}$. 
\vspace{-0.125em}\begin{romenumerate}
 \item If~$p \ge (\log n)/n$ and $t \ge \sigma$, then 
\begin{equation}\label{eq:thm:Krip}
	-\log \Pr(X \geq \mu + t)  = \Theta_{r, \xi}\bigl(\Psi(t)\bigr) \quad \text{ with } \quad \Psi(t) := \min\Bigl\{t^2/\sigma^2, \; M(t) \log (e/p)\Bigr\} . 
\end{equation}
 \item If~$\mu \ge \xi$ and~$t > 0$ satisfies $t^2/\sigma^2 \ge M(t) \log (e/p) \cdot  (\log n)^{2r}$, then 
\begin{equation}\label{eq:thm:Kripc}
	-\log \Pr(X \geq \mu + t)  = \Theta_{r, \xi}\bigl(M(t) \log (e/p)\bigr) \quad \text{ with } \quad M(t) := \max\Bigl\{ t^{1/r}, t/n^{r-1} \Bigr\}  . 
\end{equation} 
\end{romenumerate}
\end{theorem}
\noindent
By Remark~\ref{rem:variance}, inequalities~\eqref{eq:thm:Krip}--\eqref{eq:thm:Kripc} provide further evidence for Conjecture~\ref{conj:Kri} (and verify it for~$p \ge (\log n)/n$).

\subsection{Some comments}
The main focus of this paper are upper bounds on the upper tail~$\Pr(X \geq (1+\eps)\mu)$. 
Developing~\cite{APUT,LW16}, here our high-level proof strategy is based on the idea that (after ignoring certain `bad' events with negligible probabilities) using combinatorial arguments we can find a `well-behaved' subgraph~$G_0 \subseteq G_{n,p}$ in the sense that 
(i)~the number of stars~$K_{1,r}$ in~$G_0$ and~$G_{n,p}$ are approximately the same (differ by at most $\eps\mu/2$, say), and 
(ii)~the maximum degree of~$G_J$ is `not too large' (which intuitively helps for showing concentration of stars). 
Using modern Chernoff-like upper tail bounds, we then show that it is very unlikely to have a `bad' subgraph~$G' \subseteq G_{n,p}$ with `not too large' maximum degree and `many'~stars (at least~$(1+\eps/2)\mu$ many, say). 
Putting things together, the punchline is then that we can only have~$X \geq (1+\eps)\mu$ many stars if one of the discussed unlikely `bad' events occur, 
which (after some technical work) eventually gives the desired upper bounds on the upper tail~$\Pr(X \geq (1+\eps)\mu)$; see Section~\ref{sec:upper} for more details.

Finally, let us briefly compare our upper tail results for stars with very recent results from the large deviation theory literature, which are spearheaded by Chatterjee, Dembo, Lubetzky, Varadhan, Zhao, and many others (see, e.g.,~\cite{CV2011,LZ2012,CD2014,LZ2014,Eldan16,BGLZ17,DC2018}). 
For general~$H$, these aim at determining the asymptotics of~$-\log \Pr(X_H \geq (1+\eps)\E X_H)$ for constant~$\eps$ and large edge-probabilities of form~$p=\Theta(1)$ or~$p \ge n^{-\delta_H}$. 
For stars~$H=K_{1,r}$, inequality~\eqref{eq:thm:Kri} from Theorem~\ref{thm:Kri} is~weaker in the sense that it only determines~$-\log \Pr(X_{K_{1,r}} \geq (1+\eps)\E X_{K_{1,r}})$ up to constant factors, 
but it~is~stronger in the sense that it covers a much wider range of the parameters, including~$\eps=\eps(n) \ge n^{-\alpha}$ and all~$p=p(n)$ of interest.  
Obtaining such tail estimates with increased ranges of applicability is useful for combinatorial applications, where one is usually `willing to give up a little bit on the tail', in particular on the `inessential numerical constants' in the exponent (see~\cite{Vu2001,DL}).  
Furthermore, estimates of form~\eqref{eq:thm:Krip}--\eqref{eq:thm:Kripc} are also quite satisfactory from a concentration inequality perspective. 
Overall, we hope that our results stimulate more research into such estimates for other graphs~$H$.

\subsection{Organization}
%
In~Section~\ref{sec:upper} we prove the upper bounds on the upper tail  
from Theorem~\ref{thm:Krp},~\ref{thm:Kri},~and~\ref{thm:Kripc} (and discuss a simple extension). 
The corresponding (fairly routine) lower bounds are then established in Appendix~\ref{sec:lower}.

\section{Upper bounds on the upper tail}\label{sec:upper}
In this section we establish the upper bounds on the upper tail~$\Pr(X \geq (1+\eps)\mu)$ from Theorem~\ref{thm:Krp}, \ref{thm:Kri}, and~\ref{thm:Kripc}. 
Our core argument has two strands. 
In the first \emph{combinatorial part} we iteratively decrease the maximum degree of the random graph~$G_{n,p}=G_J \supseteq \cdots \supseteq G_0$ by edge-deletion (the idea is to remove large stars~$K_{1,D_j}$ with~$D_j \gg r$ from~$G_j$) 
until the final graph~$G_0$ has sufficiently low maximum degree, say at most~$D$. 
This degree bound allows us to estimate the number of stars~$K_{1,r}$ in~$G_0$ via a `well-behaved' auxiliary random variable~$X_D$. 
Taking into account the number of stars~$K_{r}$ which are removed when passing from~$G_{n,p}=G_{J}$ to~$G_0$, 
this allows us (by means of a technical event~$\cT$) 
to approximate the number~$X=X_{r,n,p}$ of copies of~$K_{1,r}$ in~$\Gnp$ using $X_D$ and several further auxiliary random variables~$N_{D_j}$ (which intuitively bound the number of~$K_{1,D_j}$ in~$G_{n,p}$). 
In the second \emph{probabilistic part} we then estimate the upper tails of these auxiliary variables using a concentration inequality of Warnke~\cite{APUT} and ad-hoc arguments (exploiting the careful definitions of the variables~$X_D$ and~$N_{D_j}$ given in Section~\ref{sec:core}). 
Putting things together, the core argument then proceeds roughly as follows: 
by the combinatorial part~$X \ge (1+\eps)\mu$ can only happen if at least one of the auxiliary variables~$X_D$ or~$N_{D_j}$ is `large', 
and by the probabilistic part  the probability of this `bad' event is at most the desired `correct' upper tail probability 
(for suitable choices of the degree~constraint~$D$ and other~parameters).

In Section~\ref{sec:core} we first illustrate this argument for the simpler setup of Theorem~\ref{thm:Krp}, 
and in  Section~\ref{sec:coreext} we then extend the argument to the more precise tail estimates of Theorem~\ref{thm:Kri} and~\ref{thm:Kripc}. 
Finally, in Section~\ref{sec:ext} we also briefly discuss a straightforward extension (to a certain sum of iid variables).

\subsection{Core argument for Theorem~\ref{thm:Krp}}\label{sec:core}
We start by introducing the main random variables and events for Theorem~\ref{thm:Krp} 
(as we shall see, their careful definitions will facilitate the interplay between the combinatorial and probabilistic parts of our argument).   
For~$x \ge 0$, let~$X_x$ denote the maximum number of copies of~$K_{1,r}$ in any subgraph~$H \subseteq \Gnp$ with maximum degree at most~$x$.
For~$y > 0$, let~$N_{y}$ denote the maximum size of any collection of edge-disjoint~$K_{1,\ceil{y}}$ in~$\Gnp$. 
For~$\beta,D,t > 0$ let~$\cT=\cT(\beta,D,t)$ denote the `technical' event~that 
\begin{equation}
\label{eq:Nj}
N_{D_j} < \frac{\beta M}{D_j} \qquad \text{for all $j \in \NN = \{0, 1, \dots\}$ ,} 
\end{equation}
where we tacitly used the following convenient parametrization:
\begin{equation}\label{eq:par}
\begin{split}
M =M(t) &:= \max\bigl\{t^{1/r}, \: t/n^{r-1}\bigr\},\\
D_j =D_j(D) &:=2^j D.
\end{split}
\end{equation}
(In this subsection we shall only use~$t=\eps\mu$; working with general~$t$ is convenient for the later extensions.)

The following combinatorial lemma is at the heart of our argument, 
and it intuitively states that $X \approx X_{D}$ whenever the event~$\cT=\cT(\beta,D,t)$ holds. 
Its  proof is inspired by ideas developed in~\cite{APUT,LW16}, but contains several new ideas. 
For example, instead of iteratively sparsifying an auxiliary hypergraph (which encodes the edge-sets of all stars~$K_{1,r}$ in~$G_{n,p}$) we here iteratively sparsify the random graph~$G_{n,p}$ itself. 
Furthermore, in order to obtain the correct tail behaviour, in inequality~\eqref{eq:Nj} we need to work with~$M=\max\{t^{1/r},t/n^{r-1}\}$ instead of the simpler choice~$M=t^{1/r}$ suggested by~\cite{APUT} 
(we achieve this by adding an extra degree bound to the argument, bounding the initial maximum degree by~$\oM=\min\{M,n\}$ instead of just~$M$). 
\begin{lemma}\label{lem:det}
Given~$\beta \in (0,1/32]$ and~$D,t>0$, the event $\cT(\beta,D,t)$ implies~$X_{D} \le X \le X_{D} + t/2$. 
\end{lemma}
\noindent 
The lower bound~$X \ge X_{D}$ of Lemma~\ref{lem:det} is trivial. 
For the upper bound the idea is to iteratively decrease the maximum degree of $\Gnp$, yielding $\Gnp=G_J \supseteq \cdots \supseteq G_0$. 
By bounding the number of~$K_{1,r}$ which are removed when passing from $G_{j+1}$ to $G_j$, this eventually allows us to estimate the total number of~$K_{1,r}$. 
\begin{proof}[Proof of Lemma~\ref{lem:det}]
Define~$\oM:=\min\{M,n\}$. 
Let~$J$ be the smallest integer~$J \ge 0$ with~$D_{J} \ge \oM$. 
We set~$G_J = \Gnp$ and inductively construct~$G_J \supseteq \cdots \supseteq G_0$. Given $G_{j+1}$, $0 \le j \le J - 1$, let $\cC_{j+1}$ be a maximal set of edge-disjoint collection of stars $K_{1, \ceil{D_j}}$. 
We remove all edges from~$G_{j+1}$ which are incident to a centre vertex of some star in~$\cC_{j+1}$, and denote the resulting graph by~$G_{j}$. 

Writing~$\Delta_j=\Delta(G_j)$ for the maximum degree of~$G_j$, 
we claim that $\Delta_j \le D_{j}$ for all $0 \le j \le J$. 
For $G_J=\Gnp$ we use a case distinction. 
If~$M \ge n$, then trivially $\Delta_{J} \le n = \oM \le D_J$. 
Otherwise~$D_J \ge \oM = M$, in which case~\eqref{eq:Nj} entails $N_{D_J} < \beta < 1$, so~$G_{n,p}=G_J$ contains no~$K_{1,\ceil{D_J}}$, and $\Delta_{J} \le \ceil{D_J}-1 \le D_J$ follows.  
Further considering~$G_{j+1}$ with $0 \le j \le J - 1$, we note that $\Delta_j \le \ceil{D_j}-1 \le D_j$ by construction, because otherwise we could add another~$K_{1,\ceil{D_j}}$ to~$\cC_{j+1}$ (contradicting maximality).  

With~$G_J \supseteq \cdots \supseteq G_0$ in hand, we now count the total number of copies of~$K_{1,r}$ in~$\Gnp=G_J$. 
Note that, given an edge $e=\{v_1,v_2\}$ of $G_{j+1}$ with $0 \le j < J$, we can construct any $K_{1,r}$ in $G_{j+1}$ containing $e$ by first selecting a centre vertex $v_c \in \{v_1,v_2\}$ and then $r-1$ additional neighbours of~$v_c$. 
Hence in $G_{j+1}$ any edge is contained in at most $2 \binom{\Delta_{j+1}}{r-1} \le 2^r D_{j}^{r-1}/(r-1)! \le 4 D_j^{r-1}$ copies of $K_{1,r}$. 
Recalling the definition of~$N_{D_j}$, note that when, passing from $G_{j+1}$ to $G_{j}$, we remove at most $N_{D_j}\Delta_{j+1} \le 2N_{D_j} D_j$ edges. 
So, since~$G_{0}$ contains at most~$X_{D_0} = X_D$ copies of $K_{1,r}$, using~\eqref{eq:Nj} and~$\max_{0 \le j < J}D_j \le \oM$ it follows that  
\begin{equation}\label{eq:x:Decomp}
X  \le X_{D} + \sum_{0 \le j < J} \bigl(2N_{D_j} D_j \cdot 4 D_j^{r-1}\bigr) 
\le X_{D} + 8\beta M \cdot \sum_{\substack{j \in \NN: D_j \le \oM}} D_j^{r-1}.
\end{equation}
Recalling~$D_j=2^{j}D$ and~$r \ge 2$, using~$\oM=\min\{M,n\}$, $M=\max \{t^{1/r}, \: t/n^{r-1}\}$ and~$\beta \le 1/32$ we infer
\begin{equation}\label{eq:x:Decomp2}
X - X_{D}  \le 8\beta M \cdot 2 \oM^{r-1} \le 16 \beta \cdot \min\{M^r, M n^{r-1}\} \le t/2, 
\end{equation}
which completes the proof. 
\end{proof}
Applying Lemma~\ref{lem:det} with~$t=\eps \mu$, in the probabilistic part of the argument 
it remains to estimate~$\Pr(X_{D} \ge \mu + \eps\mu/2)$ and~$\Pr(\neg\cT(\beta,D,\eps \mu))$. 
We shall exploit the maximum degree constraint of~$X_D$ via the following upper tail inequality of Warnke~\cite{APUT}, 
which extends classical Chernoff bounds to random variables with `well-behaved dependencies'  
(and allows us to go beyond the method of typical bounded differences~\cite{TBD}). 
\begin{theorem}[Corollary of~{\cite[Theorem~9]{APUT}}]\label{thm:C}
Let $(\xi_i)_{i \in \fS}$ be a finite family of independent random variables with~$\xi_i \in \{0,1\}$. 
Given a family $\cI$ of subsets of $\fS$, consider
random variables $Y_{\alpha} := \prod_{i \in \alpha}\xi_i$ with~$\alpha \in \cI$, and suppose $\sum_{\alpha \in \cI} \E Y_{\alpha} \le \mu$. 
Define $Z_C :=\max \sum_{\alpha \in \cJ} Y_{\alpha}$, where the maximum is taken over 
all $\cJ \subseteq \cI$ with $\max_{\beta \in \cJ}|\{\alpha \in \cJ: \alpha \cap \beta \neq \emptyset\}| \leq C$. 
Set $\varphi(x):=(1+x)\log(1+x)-x$.  
Then, for all $C,t>0$, 
\begin{equation}\label{eq:C}
\Pr(Z_C \geq \mu +t) 
\le \exp\left(-\frac{\varphi(t/\mu)\mu}{C} \right)  
\le \exp\biggl(-\frac{t^2}{2C(\mu +t)}\biggr) .
\end{equation} 
\end{theorem}
%
%
\noindent 
The main observation is that, in every subgraph~$H \subseteq G_{n,p}$ with maximum degree at most~$D$, 
any star~$K_{1,r}$ shares edges with~$O(D^{r-1})$ other stars. 
For~$X_D$ this allows us to \emph{routinely} apply Theorem~\ref{thm:C} with Lipschitz-like parameter~$C=O(D^{r-1})$, 
making inequality~\eqref{eq:XD} plausible.  
For Theorem~\ref{thm:Krp} the crux is that our choice of~$D$ will ensure~$\mu/D^{r-1} = \Theta_r(\Phi)$, 
so~\eqref{eq:XD} suggests that~$X_D \le \mu + \eps\mu/2$ fails with probability at most~$e^{-\Omega_{r,\eps}(\Phi)}$.  
\begin{corollary}\label{cor:XD}
For all $n \ge 1$, $p \in(0,1]$ and $D,t > 0$ we have 
\begin{equation}\label{eq:XD}
\Pr(X_{D} \ge \mu + t/2) \le 
\exp\left(-\frac{\varphi(t/\mu)\mu}{16D^{r-1}}\right) \leq  \exp\left(-\frac{\min\{t^2/\mu,t\}}{48D^{r-1}}\right). 
\end{equation}
\end{corollary}
\begin{proof}
Let~$\cK_{1,r}(G)$ contain all edge-subsets of~$G$ that are isomorphic to~$K_{1,r}$. 
Writing~$Y_\alpha:=\indic{\alpha \subseteq E(\Gnp)}$, there is a subgraph~$H \subseteq \Gnp$ with maximum degree at most~$\floor{D}$ such that~$X_{D} = \sum_{\alpha \in \cJ} Y_{\alpha}$ for ~$\cJ := \cK_{1,r}(H)$. 
Given $\beta \in \cJ$, we construct all edge-intersecting stars $\alpha \in \cJ$ as in the proof of Lemma~\ref{lem:det}, and infer 
\begin{equation}\label{eq:XD:cond}
\max_{\beta \in \cJ} |\{\alpha \in \cJ: \alpha \cap \beta \neq \emptyset\}| \le r \cdot 2 \binom{\floor{D}}{r-1} \le \frac{2rD^{r-1}}{(r-1)!} \le 4 D^{r-1} = : C.
\end{equation}
It follows that~$X_{D} \le Z_C$, where $Z_C$ is defined as in Theorem~\ref{thm:C} with $\cI=\cK_{1,r}(K_n)$. 
It is well-known (and easy to check by calculus) 
that for~$x \ge 0$ we have 
\begin{equation}\label{eq:varphi}
\varphi(x/2) \ge \varphi(x)/4 \qquad \text{ and } \qquad x^2 \ge \varphi(x) \ge \min\{x,x^2\}/3. 
\end{equation}
Putting things together, using Theorem~\ref{thm:C} and \eqref{eq:varphi} it follows that 
\begin{equation}\label{eq:XD:ineq}
	\Pr(X_{D} \ge \mu + t/2) \le \Pr(Z_C \ge \mu + t/2) \le \exp\left(-\frac{\varphi(t/\mu)\mu}{4C} \right) \le \exp\left(-\frac{\min\{t,t^2/\mu\}}{12C}\right) ,
\end{equation}
which completes the proof of~\eqref{eq:XD} by choice of~$C$ (see~\eqref{eq:XD:cond} above). 
\end{proof}
\noindent 
We shall estimate~$\Pr(\neg\cT(\beta,D,\eps \mu))$ via a union bound argument and   
the following upper tail estimate for~$N_{D_j}$. 
The technical assumption~\eqref{eq:ND:cond} intuitively 
ensures that vertices with degree at least~$D$ are unlikely. 
For Theorem~\ref{thm:Krp} the crux is that our choice of~$D$ will also ensure~$np/(e D_j) \le p^{\Omega(1)}$, 
so applications of inequality~\eqref{eq:ND} with~$x=\beta M/D_j$ 
suggest that~$\cT$ and thus~\eqref{eq:Nj} fails with probability at most~$n \cdot n^{-3} p^{\Omega(M)} \le n^{-2} p^{\Omega_{\eps}(\Phi)}$, say.
\begin{lemma}\label{lem:ND}
For all~$n \ge 1$, $p \in(0,1]$, and~$D > 0$ satisfying
\begin{equation}\label{eq:ND:cond}
\bigl(e^3np/D\bigr)^{D} \le n^{-8} 
\end{equation}
the following holds. For all~$x > 0$ we have 
\begin{equation}\label{eq:ND}
\Pr(N_{D_j} \ge x) \le \frac{1}{n^3} \left(\frac{np}{e \ceil{D_j}}\right)^{x D_j/2}\indic{D_j \le n} . 
\end{equation}
\end{lemma}
\begin{proof}
As $\binom{m}{z} \le (me/z)^z$ for all integers~$m,z \ge 1$, 
by exploiting the disjointness condition of~$N_{D_j}$ we infer 
\begin{equation}\label{eq:NDjx:1}
\Pr(N_{D_j} \ge x) \le n^{\ceil{x}}\binom{n}{\ceil{D_j}}^{\ceil{x}} p^{\ceil{x}\ceil{D_j}} \le \left(n \left(\frac{enp}{\ceil{D_j}}\right)^{\ceil{D_j}}\right)^{\ceil{x}} 
\end{equation}
As the function~$x \mapsto (e^3np/x)^x$ is decreasing for~$x \ge e^2np$, and~\eqref{eq:ND:cond} implies $\ceil{D_j} \ge D \ge e^3np$, we deduce  
\begin{equation*}
	\left(\frac{enp}{\ceil{D_j}}\right)^{\ceil{D_j}} = \left(\frac{e^3np}{\ceil{D_j}}\right)^{\frac{\ceil{D_j}}{2}}\left(\frac{np}{e\ceil{D_j}}\right)^{\frac{\ceil{D_j}}{2}} \le \left(\frac{e^3np}{D}\right)^{\frac{D}{2}} \cdot \left(\frac{np}{e\ceil{D_j}}\right)^{\frac{\ceil{D_j}}{2}} \le n^{-4} \left(\frac{np}{e\ceil{D_j}}\right)^{\ceil{D_j}/2}.
\end{equation*}
Plugging this into \eqref{eq:NDjx:1} readily establishes inequality~\eqref{eq:ND}, since trivially~$N_{D_j}=0$ when~$D_j > n$. 
\end{proof}
For the proof of the upper bound of Theorem~\ref{thm:Krp} it remains to pick suitable~$D$, i.e.,  
which satisfies the technical assumption~\eqref{eq:ND:cond} and yields the `correct' exponent in~\eqref{eq:XD} and suitable applications of~\eqref{eq:ND}. 
\begin{proof}[Proof of the upper bound in~\eqref{eq:thm:Krp} of Theorem~\ref{thm:Krp}]
For concreteness, define~$\beta := 1/32$ and~$\gamma := 1/(16r)$, as~well~as  
\[
A := \max\left\{e^4, \; 8/\gamma\right\}, \quad s := \log(e/p^{\gamma}) , \quad \text{and} \quad D := A \cdot \max\left\{1, \; \frac{\min\{\mu^{1/r}, n\}}{s^{1/(r-1)}}\right\} .
\]
For later reference, we record that there is a constant~$d=d(r) >0$ such that, for $n \ge n_0(r)$,  
\begin{equation}\label{eq:mu:lb}
d n^{r+1}p^r \le \mu \le n^{r+1}p^r .
\end{equation}
By Lemma~\ref{lem:det}, the upper tail of the number~$X=X_{r,n,p}$ of~$K_{1,r}$-copies satisfies  
\begin{equation}\label{eq:thm:Krp:Pr:0}
\Pr(X \ge (1+\eps)\mu) \le \Pr(X_D \ge \mu+\eps\mu/2) + \Pr(\neg\cT(\beta,D,\eps \mu)) .
\end{equation}
Gearing up to bound~$\Pr(\neg\cT(\beta,D,\eps \mu))$ via Lemma~\ref{lem:ND}, 
using~$e = p^{\gamma}e^s$ and inequality~\eqref{eq:mu:lb} together with the bound~$s^{1/(r-1)} \le s = 1 + \log p^{-\gamma} \le p^{-\gamma}$ (as~$1+x \le e^{x}$) it follows that 
\begin{equation}\label{eq:win:logfactor}
\frac{np}{eD}= 
\frac{np^{1-\gamma} e^{-s}}{D} \le \indic{p \le n^{-1/(1-\gamma)}} \frac{e^{-s}}{A} + \indic{p > n^{-1/(1-\gamma)}}\frac{e^{-s}}{A \min\{d^{1/r}n^{1/r}p^{2\gamma}, \: p^{2\gamma-1}\}} 
\le \frac{e^{-s}}{A} 
\le e^{-s} ,
\end{equation}
where here and below we shall always tacitly assume $n \ge n_0(r,d)$ whenever necessary. 
Since the above calculation also gives~$D \ge A np^{1-\gamma}$, 
together with $D \ge A$ it follows that  
\[
\bigl(e^3np/D\bigr)^{D} \le (p^{\gamma}/e)^{D} \le \indic{p \le n^{-1}}n^{-A\gamma} + \indic{p > n^{-1}}e^{-Anp^{1-\gamma}} \le n^{-8} .
\]
Applying a union bound argument, using estimates~\eqref{eq:ND}, $D_j = 2^jD \ge D$, and~\eqref{eq:win:logfactor} it follows that
\begin{equation}\label{eq:T:fails}
\begin{split}
\Pr(\neg\cT(\beta,D,\eps\mu)) & \le \sum_{j \in \NN} \Pr(N_{D_j} \ge \beta M/D_j) 
 \le n \cdot \frac{1}{n^3} \left(\frac{np}{e D}\right)^{\beta M/2} \le \frac{1}{n^2} \cdot e^{-\beta Ms/2} . 
\end{split}
\end{equation}
Recalling~\eqref{eq:thm:Krp:Pr:0} and the definition of~$M=M(\eps\mu)$, by applying Corollary~\ref{cor:XD} with~$t := \eps\mu$ 
it follows that there is a constant $c=c(\beta,A,\gamma,r)>0$ and suitable parameters~$\zeta,\Pi>0$ such that 
\begin{equation}\label{eq:thm:Krp:Pr}
\begin{split}
\Pr(X \ge (1+\eps)\mu) 
& \le \exp\left(-\frac{\min\{\eps,\eps^2\}\mu}{48D^{r-1}}\right) + \frac{1}{n^2}\exp\left(-\frac{\beta Ms}{2}\right) \\
& \le (1+n^{-2}) \cdot \exp\Bigl(-c \underbrace{\min\{\eps, \: \eps^2, \: \eps^{1/r}\}}_{=: \zeta} \underbrace{\min\bigl\{\mu, \; \max\bigl\{\mu^{1/r},\: \mu/n^{r-1}\bigr\}s\bigr\}}_{=: \Pi}\Bigr) .
\end{split}
\end{equation}
We find the above upper tail estimate very satisfactory, but in the literature it has become standard 
to suppress multiplicative factors such as~$1+n^{-2}$ in~\eqref{eq:thm:Krp:Pr}, which is straightforward when~$c \zeta\Pi \ge 1$ holds (rescaling the exponent~$c \zeta\Pi$ by a factor of~$1/2$, say).  
In the remaining case~$1 > c \zeta\Pi$ Markov's inequality gives 
\[
\Pr(X \ge (1+\eps)\mu) \le \frac{1}{1+\eps} = 1 - \frac{\eps}{1+\eps} \le \exp\left(-\frac{\eps}{1+\eps}\right) \le \exp\Bigl(-\tfrac{c}{2}\min\{\eps,1\}\zeta \Pi\Bigr) .
\]
Finally, noting~$s = \log(e/p^{\gamma}) \ge \log(1/p^{\gamma}) = \gamma \log(1/p)$ then establishes the upper bound in~\eqref{eq:thm:Krp}. 
\end{proof}

\subsection{Extension of the argument to Theorem~\ref{thm:Kri} and~\ref{thm:Kripc}}\label{sec:coreext}
We now extend the arguments from Section~\ref{sec:core} to the upper bounds of Theorem~\ref{thm:Kri} and~\ref{thm:Kripc}. 
To obtain sub-Gausssian decay~$\varphi(\eps)\mu^2/\sigma^2$ in the exponent of tail-inequality~\eqref{eq:XD} for~$X_D$, 
in view of the well-known variance estimate~$\sigma^2 = \Theta_{r,\xi}((1+(np)^{r-1})\mu)$ from Remark~\ref{rem:variance} we here would like to pick~$D=\Theta_{r,\xi}(1+np)$ for some range of~$t=\eps\mu$. 
However this choice causes a major problem:\footnote{For~$D=\Theta_{r,\xi}(1+np)$ another problem is that the technical assumption~\eqref{eq:ND:cond} of Lemma~\ref{lem:ND} then breaks when~$np$ is close to~one, which partially explains why in the upcoming Theorem~\ref{thm:Kr} we shall exclude fairly small~$t$ when~$np \in (n^{-\gamma}, \gamma \log n)$.}
in the key estimate~\eqref{eq:win:logfactor} we can no longer win an extra log-factor (via $np/(eD) \le e^{-s}$) when we bound the~$N_{D_j}$ variables using~\eqref{eq:ND} from Lemma~\ref{lem:ND}.  
Our strategy for overcoming this obstacle is to refine the technical event~$\cT=\cT(\beta,D,t)$, by enforcing different upper bounds on~$N_{D_j}$ when~$D_j=2^jD$ is small (so that in the probabilistic arguments we automatically win an extra logarithmic factor, without destroying the combinatorial counting arguments from Lemma~\ref{lem:det}). 

Turning to the details, for~$\gamma, \beta, D,t > 0$ let~$\cTp=\cTp(\gamma,\beta,D,t)$ denote the `technical' event that 
\begin{alignat}{2}
\label{eq:Njs:mod}
N_{D_j} &< \frac{\beta Ms}{D_j} &\qquad&\text{for all $j \in \NN$ with $D_j < \min\{M,n\}/s^{1/(r-1)}$, \ and} \\
\label{eq:Nj:mod}
N_{D_j} &< \frac{\beta M}{D_j} &&\text{for all $j \in \NN$ with $D_j \ge \min\{M,n\}/s^{1/(r-1)}$,} 
\end{alignat}
where, in addition to the parameters~$M =\max\{t^{1/r}, t/n^{r-1}\}$ and $D_j = 2^j D$ from~\eqref{eq:par}, we tacitly used 
\begin{equation}\label{eq:par2}
s = s(\gamma) := \log(e/p^{\gamma}). 
\end{equation}
\begin{lemma}\label{lem:det2}
Given~$\beta \in (0,1/64]$ and~$\gamma,D,t > 0$, the event $\cTp(\gamma,\beta,D,t)$ implies~$X_{D} \le X \le X_{D} + t/2$. 
\end{lemma}
\begin{proof}
The proof of Lemma~\ref{lem:det} carries over, except for the final inequalities~\eqref{eq:x:Decomp}--\eqref{eq:x:Decomp2} that bound~$X$ from above.  
Recalling that~$\oM = \min \{M, n\}$, by mimicking the argument leading to~\eqref{eq:x:Decomp} we here obtain
\[
X - X_{D}\le \sum_{0 \le j < J} \bigl(2N_{D_j} D_j \cdot 4 D_j^{r-1}\bigr) 
\le 8\beta M \cdot \biggl(s\sum_{\substack{j \in \NN: D_j < \oM/s^{1/(r-1)}}}D_j^{r-1} + \sum_{\substack{j \in \NN: \oM/s^{1/(r-1)} \le D_j \le \oM}} D_j^{r-1}\biggr) .
\]
Recalling~$D_j=2^{j}D$ and $r \ge 2$, using~$\beta \le 1/64$ it then follows similarly to~\eqref{eq:x:Decomp}--\eqref{eq:x:Decomp2} that 
\[
X - X_{D}  
\le 8\beta M \cdot 4 \oM^{r-1} \le 32 \beta \cdot \min\{M^r, M n^{r-1}\} \le t/2, 
\]
which completes the proof. 
\end{proof}

We are now ready to prove the following slightly more general upper tail estimate for the number~$X=X_{r,n,p}$ of~$K_{1,r}$-copies in~$G_{n,p}$, 
which (as we shall see) implies the upper bounds in Theorems~\ref{thm:Kri} and~\ref{thm:Kripc}.
\begin{theorem}[Upper tail bounds: technical result]\label{thm:Kr}
Given $r \geq 2$, let~$X=X_{r,n,p}$. Set $\mu:=\E X$, $\Lambda := \mu (1+(np)^{r-1})$, and $\varphi(x):=(1+x)\log(1+x)-x$. 
Given $\gamma > 0$, suppose that either 
\begin{equation*}
np\not\in(n^{-\gamma}, \gamma \log n)
\quad \text{ or } \quad 
	t^2/\mu \ge \indic{t \le \min\{\mu,n^r\}} \gamma  \min\bigl\{t^{1/r}(\log n)^r, \: Ms (\log n)^{r-1}\bigr\}
\end{equation*}
holds, 
where the parameters~$M$ and~$s$ are defined in~\eqref{eq:par} and~\eqref{eq:par2}.
Then we have 
\begin{equation}\label{eq:thm:Kr:U}
	\Pr(X \geq \mu+t)  \leq (1+n^{-1}) \cdot \exp\biggl(- \Omega_{r,\gamma} \Bigl( \min\Bigl\{\varphi(t/\mu)\mu^2/\Lambda, \; M \log (e/p)\Bigr\}\Bigr) \biggr).
\end{equation}
\end{theorem}
\begin{proof} 
Let~$\beta := 1/64$. 
We distinguish the following three cases: (i)~$np \ge \gamma \log n$, (ii)~$np \le n^{-\gamma}$, and (iii)~$t^2/\mu \ge \indic{t \le \min\{\mu,n^r\}} \gamma  \min\{t^{1/r}(\log n)^r, Ms (\log n)^{r-1}\}$. 
Note that in all three cases we may assume~$\gamma \le 1/(16r)$, since decreasing~$\gamma$ yields a less restrictive assumption. Furthermore, in case~(iii) we may also assume that  $n^{-\gamma} \le np \le \log n$ holds (otherwise case~(i) or~(ii) apply).  
For concreteness, define
\[
	A:=\max\left\{e^4, \; 8 \cdot (3/\gamma)^{1/(r-1)}, \; 8/\gamma \right\} \quad  \text{and} \quad D := A \cdot \max\left\{1+np,\; \left(\frac{\varphi(t/\mu)\mu}{Ms}\right)^{1/(r-1)}\right\}  .
\]
(We remark that in cases~(i)--(ii) the simpler choice~$D = A (1+np)$ suffices.) 
We defer the somewhat technical proofs of the following claims regarding Lemma~\ref{lem:ND}:~(a)~assumption~\eqref{eq:ND:cond} holds, and (b)~inequality~\eqref{eq:ND} implies
\begin{equation}\label{eq:T:bounds}
	\Pr\bigl(\neg\cTp(\beta,\gamma,D,t)\bigr) \le \frac{1}{n} \max\Bigl\{e^{-\beta Ms/2}, \: e^{-\Psi} \Bigr\} \qquad \text{with} \qquad \Psi := \varphi(t/\mu)\mu^2/\Lambda,
\end{equation}
where here and below we shall again tacitly assume~$n \ge n_0(r)$. 
Analogously to inequalities~\eqref{eq:thm:Krp:Pr:0} and~\eqref{eq:thm:Krp:Pr}, by first applying Lemma~\ref{lem:det2} and Corollary~\ref{cor:XD}, 
and then using~$(1+np)^{r-1} = \Theta_r(\Lambda/\mu)$, it follows that 
\begin{equation*}\label{eq:Kr:E}
\begin{split}
\Pr(X \ge \mu+t) 
& \le \exp\left(-\frac{ \varphi(t/\mu)\mu}{16D^{r-1}}\right) + \textstyle\frac{1}{n}\exp\Bigl(-\frac{\beta}{2} \min\bigl\{\Psi, \: Ms\bigr\}\Bigr)
\le (1+n^{-1}) \cdot \exp\Bigl(-\Omega_{r,\gamma}\left( \min\bigl\{\Psi, \: Ms\bigr\}\Bigr) \right).
\end{split}
\end{equation*}
Since~$s = \log(e/p^{\gamma}) \ge \gamma \log(e/p)$, this establishes inequality~\eqref{eq:thm:Kr:U}. 

It remains to verify claims~(a)~and~(b) above, and start with claim~(a), i.e., that the assumption~\eqref{eq:ND:cond} of Lemma~\ref{lem:ND} holds. 
Note that $D \ge A (1+np) \ge e^4 np$. 
Furthermore, in case~(i) we have~$D \ge A\gamma \log n$, and in case~(ii) we have~$np \le n^{-\gamma}$ and~$D \ge A$. 
So, in both cases, using $A \ge 8/\gamma$ we infer
\begin{equation}\label{eq:Kr:c12}
\bigl(e^3np/D\bigr)^{D} \le \min\bigl\{e^{-D},(np)^{D}\bigr\} \le n^{-A\gamma} \le n^{-8} .
\end{equation}
Proceeding analogously, in the cumbersome case~(iii) it suffices to show~$D \ge 8 \log n$. 
Using~$\gamma \le 1/(16r)$, $p \ge n^{-1-\gamma}$ and~\eqref{eq:mu:lb}, it is routine to see that~$s \le \log n$ and~$\mu \ge n^{1/2}$.  
Assuming~$t \ge \mu$, by first using~\eqref{eq:varphi} and then distinguishing the cases~$t \ge n^{r}$ (where~$M=t/n^{r-1}$) and~$t \le n^{r}$ (where~$M=t^{1/r}$), it follows that 
\begin{equation}\label{eq:casedist}
	D^{r-1} \ge \frac{A^{r-1}\varphi(t/\mu)\mu}{Ms} \ge \frac{A^{r-1}t}{3 Ms} \ge \frac{A^{r-1}\min\bigl\{n^{r-1},\:  \mu^{1-1/r}\bigr\}}{3\log n} \ge (A\log n)^{r-1} .
\end{equation}
Assuming~$t \le \mu$, we note that assumption $p \le (\log n)/n$ implies $\mu \le n^r$ (hence~$t \le n^{r}$ and thus~$M=t^{1/r}$, as noted above). Hence, by first using~\eqref{eq:varphi} and then the assumed lower bound on~$t$ from case~(iii), we infer
\[
D^{r-1} \ge \frac{A^{r-1}\varphi(t/\mu)\mu}{Ms} \ge \frac{A^{r-1} t^2}{3 \mu Ms} 
\ge \frac{\gamma A^{r-1} \min \{t^{1/r}(\log n)^r, Ms (\log n)^{r-1}\}}{3  Ms}
= \gamma/3 \cdot (A\log n)^{r-1} .
\]
Each time~$D \ge 8\log n$ follows readily by definition of~$A$, establishing claim~(a), as discussed above. 

Finally, we verify claim~(b), i.e., that inequality~\eqref{eq:ND} implies estimate~\eqref{eq:T:bounds}. 
We start by observing that if $\cTp(\beta,\gamma,D,t)$  fails then a fortiori~$N_{D_0} \ge 1$.
Hence, using~\eqref{eq:ND} with~$x=1$ and~$D_0=D \ge e^3np$, we deduce   
\begin{equation}\label{eq:T:0}
	 \Pr(\neg\cTp(\beta,\gamma,D,t)) \le \Pr(N_{D_0} \ge 1) \le \frac{1}{n^3} \cdot e^{-D} . 
\end{equation}
Analogously to~\eqref{eq:T:fails}, using inequality~\eqref{eq:ND} and~$D_j = 2^jD \ge e^3np$ it also follows that 
\begin{equation}\label{eq:T:fails:1}
\begin{split}
\Pr(\neg\cTp(\beta,\gamma,D,\eps\mu)) & \le \sum_{j \in \NN: D_j \le \oM/s^{1/(r-1)}} \Pr(N_{D_j} \ge \beta M s/D_j) + \sum_{j \in \NN: D_j \ge \oM/s^{1/(r-1)}} \Pr(N_{D_j} \ge \beta M/D_j) \\
& \le \frac{1}{n^2} \cdot e^{-\beta Ms} + \frac{1}{n^2} \cdot \left(\frac{np}{e\bigceil{\oM/s^{1/(r-1)}}}\right)^{\beta M/2} .
\end{split}
\end{equation}
We now use a fairly technical case distinction to verify that the two estimates~\eqref{eq:T:0}--\eqref{eq:T:fails:1} together imply~\eqref{eq:T:bounds}. 
Assuming~$\oM \ge np^{1-2\gamma}$, analogously to the proof of~\eqref{eq:win:logfactor} we have $np s/\oM \le p^{2\gamma} s \le  p^{\gamma} = e^{1-s}$, so that 
\begin{equation}\label{eq:T:bounds:M}
	\left(\frac{np}{ e\bigceil{\oM/s^{1/(r-1)}}}\right)^{\beta M/2} \le  \left(\frac{np s}{e\oM}\right)^{\beta M/2} \le e^{-\beta Ms/2} \qquad \text{when }\oM \ge np^{1-2\gamma} .
\end{equation}
Next we assume $p \le n^{-1/(1-\gamma)}$, which implies $np/e\le p^{\gamma}/e= e^{-s}$, so that 
\begin{equation}\label{eq:T:bounds:p}
	\left(\frac{np}{ e\bigceil{\oM/s^{1/(r-1)}}}\right)^{\beta M/2} \le  \left(\frac{np}{e}\right)^{\beta M/2} \le e^{-\beta Ms/2} \qquad \text{when }p \le  n^{-1/(1-\gamma)}.
\end{equation}
In the remaining case~$\oM < np^{1-2\gamma}$ and~$p \ge n^{-1/(1-\gamma)}$ hold. 
Since $\oM < n$ implies $\oM = M$, we infer $t \le M^r= (\oM)^r \le n^rp^{r-2r\gamma}$. 
So, recalling that $\Psi \le t^2/\Lambda \le t^2/[(np)^{r-1}\mu]$ by~\eqref{eq:varphi} and that $\mu \ge dn^{r+1}p^r$ by~\eqref{eq:mu:lb}, using $D \ge np$, $p \ge n^{-1/(1-\gamma)}$ and $\gamma \le 1/(16r)$ we deduce that 
\[
	\frac{\Psi}{D} \le \frac{t^2}{(np)^r\mu} \le \frac{n^rp^{r-4r\gamma}}{\mu} \le \frac{1}{dn p^{4r\gamma}} \le \frac{1}{dn^{1-4r\gamma/(1-\gamma)}} \le \frac{1}{dn^{1/2}} \le 1 ,
\]
establishing~$D \ge \Psi$. It follows that 
\begin{equation}\label{eq:T:bounds:D}
	e^{-D} \le e^{-\Psi}  \qquad \text{when }\oM < np^{1-2\gamma} \text{ and } p \ge n^{-1/(1-\gamma)},
\end{equation}
which together with inequalities~\eqref{eq:T:0}--\eqref{eq:T:bounds:p} 
implies the claimed estimate~\eqref{eq:T:bounds}. 
\end{proof}
We now deduce the upper bounds of Theorem~\ref{thm:Kri} and~\ref{thm:Kripc} from the upper tail inequality~\eqref{eq:thm:Kr:U}. 
\begin{proof}[Proof of the upper bound in~\eqref{eq:thm:Kri} of Theorem~\ref{thm:Kri}]
Let~$\gamma:=1/(9r)$.  
For~$t := \eps \mu \ge n^{-\alpha} \mu$ and~$n \ge n_0(r)$ it is routine to check that~$t^{2-1/r}/\mu \ge \indic{np \ge n^{-\gamma}} \gamma (\log n)^r$ holds for~$\alpha=\alpha(r)>0$ sufficiently small.
Hence Theorem~\ref{thm:Kr} applies with~$t=\eps\mu$, where~$\Lambda=\Theta_{r,\xi}(\sigma^2)$ by Remark~\ref{rem:variance}. 
Using~$\Phi(\eps) \ge 1$ it follows that 
\begin{equation}\label{eq:thm:Kri:bound}
	\Pr(X \geq (1+\eps)\mu)  \leq (1+n^{-1}) \cdot e^{- \Omega_{r, \xi} ( \Phi(\eps))} \le e^{- \Omega_{r,\xi} ( \Phi(\eps) )} ,
\end{equation}
establishing the upper bound in~\eqref{eq:thm:Kri}. 
\end{proof}
\noindent 
For Theorem~\ref{thm:Kripc} we shall simplify the form of the exponent in~\eqref{eq:thm:Kr:U} via the following auxiliary result, 
writing~$a_n \asymp b_n$ instead of~$a_n = \Theta(b_n)$ for typographic reasons (the assumption~$p \ge n^{-9}$ in~(ii) is~ad-hoc). 
\begin{lemma}\label{cl:phi_to_t_square}
Given~$\xi \in (0,1)$, the following holds whenever~$p \in (0,1-\xi]$. 
\vspace{-0.5em}\begin{romenumerate}
 \item If $t \le \mu$, then 
\begin{equation}\label{eq:simpler_normal_term}
		t^2/\sigma^2 \asymp \varphi(t/\mu)\mu^2/\sigma^2 \asymp_{r,\xi} \varphi(t/\mu)\mu^2/\Lambda. 
\end{equation}
 \item If~$t \ge \mu$ and $t^{1-1/r} \ge (\log n)\indic{p < 1/n}$, then~$p \ge n^{-9}$ implies 
	\begin{equation}\label{eq:no_normal_term}
		t^2/\sigma^2 \ge \varphi(t/\mu)\mu^2/\sigma^2 \asymp_{r,\xi} \varphi(t/\mu)\mu^2/\Lambda = \Omega_{r,\xi}\bigl(M \log (e/p)\bigr). 
	\end{equation}
 \item If~$t^2/\sigma^2 \ge \min \{M, 1\}$ and~$\mu + t \ge 1$, then~$t = \Omega_{r,\xi}(1)$.
\end{romenumerate}
\end{lemma}
\begin{proof}
Inequality~\eqref{eq:simpler_normal_term} and the first two estimates of equation~\eqref{eq:no_normal_term} follow immediately from~\eqref{eq:varphi} and~$\Lambda = \Theta_{r,\xi}(\sigma^2)$, see Remark~\ref{rem:variance}. 
We now turn to the final inequality of equation~\eqref{eq:no_normal_term}. 
By combining~\eqref{eq:varphi} and $\Lambda/\mu = 1+(np)^{r-1}$ with~$M =\max\{t^{1/r}, t/n^{r-1}\}$ and~$t^{1-1/r} \ge (\log n) \indic{p < 1/n} + \mu^{1-1/r}\indic{p \ge 1/n}$, 
using~$p \ge n^{-9}$ and~$\mu^{1-1/r} = \Omega_r(n^{1/r} (np)^{r-1})$, see~\eqref{eq:mu:lb}, it follows similarly to~\eqref{eq:casedist} that 
\begin{equation*}
	\frac{\varphi(t/\mu)\mu^2/\Lambda}{M} \ge \frac{t\mu}{3\Lambda M} \ge \frac{\min\{t^{1-1/r},\: n^{r-1}\}}{6\max\{1, \: (np)^{r-1}\}} \ge \frac{1}{6}\min \left\{\log n, \; \frac{\mu^{1-1/r}}{(np)^{r-1}}, \; n^{r-1}, \; \frac{1}{p^{r-1}} \right\} = \Omega_r\bigl(\log(e/p)\bigr) ,
\end{equation*}
where we exploited that calculus gives~$p^{r-1} \log(e/p) = O_r(1)$; 
this completes the proof of claims~(i)--(ii).   

For claim~(iii) we may of course assume~$t \le 1/2$ (otherwise there is nothing to show).  
Hence~$t^2/\sigma^2 \ge \min \{M, 1\} \ge \min\{t^{1/r},1\} = t^{1/r}$ implies~$t^{2-1/r} \ge \sigma^2 = \Omega_{r,\xi}(\mu)$ by Remark~\ref{rem:variance}, 
which in turn gives~$t = \Omega_{r,\xi}(1)$, because~$t + \mu \ge 1$ and~$t \le 1/2$ together imply~$\mu \ge 1-t \ge 1/2$, completing the proof. 
\end{proof}
\begin{proof}[Proof of the upper bound in~\eqref{eq:thm:Krip} of Theorem~\ref{thm:Kripc}]
Applying Theorem~\ref{thm:Kr} (with~$\gamma = 1$), using~(i)--(ii) of Lemma~\ref{cl:phi_to_t_square} it follows that
inequality~\eqref{eq:thm:Kr:U} holds with~$\Omega_{r,\xi}(\Psi(t))$ in the exponent, 
where~$\Psi(t) \ge \min\{1, t^{1/r}\} = \Omega_{r,\xi}(1)$ by~(iii) of Lemma~\ref{cl:phi_to_t_square}. 
Absorbing the~$1 + n^{-1}$ factor similar to~\eqref{eq:thm:Kri:bound} then 
establishes the upper bound in~\eqref{eq:thm:Krip}. 
\end{proof}
\begin{proof}[Proof of the upper bound in~\eqref{eq:thm:Kripc} of Theorem~\ref{thm:Kripc}]
Since~$\sigma^2 = \Omega_{r,\xi}(\mu)$ by Remark~\ref{rem:variance}, note that the assumption 
\begin{equation}\label{eq:normal_more_cluster}
	t^2/\sigma^2 \ge M \log (e/p) \cdot (\log n)^{2r}
\end{equation}
implies~$t^2/\mu \ge M \log(e/p) \cdot (\log n)^{r-1}$, so that Theorem~\ref{thm:Kr} (with~$\gamma = 1$) applies. 
Using~\eqref{eq:normal_more_cluster}, by~(iii) of Lemma~\ref{cl:phi_to_t_square} we also infer that~$M \ge t^{1/r} =\Omega_{r,\xi}(1)$.  
Absorbing the~$1 + n^{-1}$ factor as before, 
it remains to show that the exponent of inequality~\eqref{eq:thm:Kr:U} is~$\Omega_{r,\xi}(M \log (e/p))$. 
For~$t \le \mu$ this follows from~\eqref{eq:simpler_normal_term}~of Lemma~\ref{cl:phi_to_t_square} and~\eqref{eq:normal_more_cluster}. 
For~$t \ge \mu$ this follows from~\eqref{eq:no_normal_term}~of Lemma~\ref{cl:phi_to_t_square}, since~\eqref{eq:normal_more_cluster} and~$p < n^{-1}$ imply~$t^2/(\log n)^{2r+1} \ge \sigma^2 M = \Omega_{r,\xi}(\mu)= \Omega_{r,\xi}(1)$ and thus~$t^{1-1/r} \ge (\log n)\indic{p < 1/n}$, as required. 
\end{proof}

\subsection{Straightforward extension to a certain sum of iid variables}\label{sec:ext}
We close this section by recording that minor (and in fact simpler) variants of our proofs also apply~to the following sum of independent random variables:
\begin{equation}\label{eq:X:binom}
X:=\sum_{i \in [n]} \binom{Y_i}{r} \qquad \text{ with independent~$Y_i \sim \Bi(n,p)$.} 
\end{equation}
Indeed, in view of the structural similarities to the number of $r$-armed stars in~$G_{n,p}$ (which satisfies~$X_{n,r,p}=\sum_{v \in [n]}\binom{\mathrm{d}_v}{r}$, writing~$\mathrm{d}_v$ for the degree of~$v$), here we set $X_x := \sum_{i \in [n]: Y_i \le \floor{x}}\binom{Y_i}{r}$, and define~$N_x$ as the number of~$i \in [n]$ with~$Y_i \ge \ceil{x}$. 
Now the proofs of Lemma~\ref{lem:det} and~\ref{lem:det2} carry over with minor changes: exploiting that there are no dependencies between the $Y_i$, using a simple dyadic decomposition 
we here obtain 
\[
X \le X_{D} + \sum_{0 \le j < J} \biggl[N_{D_j} \cdot \binom{\floor{D_{j+1}}}{r}\biggr] \le X_{D} + 2\sum_{0 \le j < J} N_{D_j} D_j^{r} \le \cdots \le X_{D} + t/2.
\]
For the proof of Corollary~\ref{cor:XD} it suffices to show that $X_D \le Z_C$ holds in the present setting. 
Since $Y_i$ is a sum of $n$ independent indicators $\xi_{i,j}$, we may write each $\binom{Y_i}{r}$ as a sum of $\binom{n}{r}$ dependent indicators (which each are products of some~$r$ distinct independent variables~$\xi_{i,j}$). 
Using the constraint $Y_i \le \floor{D}$ the analogous left hand side of \eqref{eq:XD:cond} is thus bounded by $r \cdot \binom{\floor{D}}{r-1} \le 2 D^{r-1}$, which in turn implies $X_D \le Z_C$, as desired. 
Since the proof of Lemma~\ref{lem:ND} also remains valid (as inequality~\eqref{eq:NDjx:1} carries over), we thus arrive at the following result. 
\begin{theorem}[Upper tail bounds: an extension]\label{thm:ext}
The upper bounds on the upper tail~$\Pr(X \geq (1+\eps)\mu)$ from Theorem~\ref{thm:Krp}, \ref{thm:Kri}, \ref{thm:Kripc}, and~\ref{thm:Kr} remain valid 
for the random variable~$X$ defined in~\eqref{eq:X:binom}. 
\end{theorem}
Perhaps surprisingly, we are not aware of any standard method or inequality (for sums of iid variables) which can routinely recover the upper tail bounds from Theorem~\ref{thm:ext}.  
Here one technical difficulty seems to be that each summand~$\binom{Y_i}{r}$ has an upper tail that decays slower than exponentially (for~$r \ge 2$), 
which presumably is closely linked to the somewhat non-standard~$\log(1/p)$~term in the exponent.

\bigskip{\noindent\bf Acknowledgements.} 
We would like to thank Svante Janson for a helpful discussion, 
and the CPC referee report from June~2015 (on an earlier version of this paper) for suggestions concerning the presentation.

\small
\bibliographystyle{plain}

\normalsize

\begin{appendix}

\section{Appendix: Lower bounds on the upper tail}\label{sec:lower}
In this appendix we establish fairly routine lower bounds on the upper tail~$\Pr(X \geq (1+\eps)\mu)$ from Theorem~\ref{thm:Krp}, \ref{thm:Kri}, and~\ref{thm:Kripc} (omitting some straightforward details).  
Following~\cite{APUT} we obtain our lower bounds via the following three events: 
that many copies of~$K_{1,r}$ `cluster' on few edges (see Lemma~\ref{lem:cls} and~\ref{lem:cls:refined}), 
that most copies of~$K_{1,r}$ arise disjointly (see Lemma~\ref{lem:disjm} and~\ref{lem:disj}), and 
that~$\Gnp$ contains more edges than expected (see~Lemma~\ref{lem:edges}).

\subsection{Basic argument for Theorem~\ref{thm:Krp}}\label{sec:core:LB}
For Theorem~\ref{thm:Krp} we shall use two different lower bounds, and the first one 
is based on the idea that relatively few edges (which `cluster' in an appropriate way) can create fairly many stars~$K_{1,r}$. 
This is formalized by the following result, which implies~$\Pr(X_{r,n,p} \ge x) \ge \Pr(F \subseteq G_{n,p}) = p^{|E(F)|}$ since~$F \subseteq G_{n,p}$ enforces~$X_{r,n,p} \ge x$. 
\begin{lemma}[Clustering]\label{lem:cls}
For every $r \ge 1$ there is $D=D(r)>0$ so that for all~$n \ge 1$ and~$0 < x \le X_{r,n,1}$ there is~$F \subseteq K_n$ with~$|E(F)| \le D\max\{x^{1/r},x/n^{r-1},1\}$ edges such that $F$ contains at least~$x$ copies of~$K_{1,r}$. 
\end{lemma}
\noindent
Inspired by the proofs of Theorem~1.3 and~1.5 in~\cite{UTSG}, the idea is to use a complete bipartite graph~$F =K_{y,z}$ with $y = \Theta_r(\min\{x^{1/r},n\})$ and $z=\Theta_r(x/y^r)$, which contains $yz = \Theta_r(x/y^{r-1})= O_r(\max\{x^{1/r},x/n^{r-1}\})$ edges and at least $z \binom{y}{r} = \Theta_r(z y^r)= \Omega_r(x)$ copies of~$K_{1,r}$ (certain border cases require minor care).  
\begin{proof}[Proof of Lemma~\ref{lem:cls}]
Let~$x_0:=2 (4r)^r$, $n_0:=(r+1)x_0$, and~$D := n_0^2$. 
If~(i)~$x_0 \le x \le n^{r+1}/D$ and~$n \ge n_0$, then we let~$F:=K_{y,z}$, with~$y:=\ceil{\min\{x^{1/r},n\}/4}$ and~$z:=\ceil{r^r x/y^r}$. 
Note that~$F \subseteq K_n$ exists, since it is easy to check that~$1 < y \le n/2$ and~$1 < z \le n/2$, say (we leave the details to the reader). 
Furthermore, $F$ contains at least~$z \binom{y}{r} \ge z (y/r)^r \ge x$ many~$K_{1,r}$,  
and~$|E(F)| = yz \le 2r^r x/y^{r-1} \le D \max\{x^{1/r},x/n^{r-1}\}$~edges. 

If either~(ii)~$1 \le n < n_0$ or~(iii)~$x > n^{r+1}/D$ and $n \ge n_0$, then we let~$F:=K_n$, which trivially contains~$X_{r,n,1} \ge x$ copies of~$K_{1,r}$, and~$|E(F)| < n^2 < \max\{n_0^2, Dx/n^{r-1}\} = D \max\{1,x/n^{r-1}\}$ edges. 

Finally, if~(iv)~$x < x_0$ and~$n \ge n_0$, then we let~$F:=K_{n_0}$, which contains at least~$n_0/(r+1) = x_0 > x$ vertex disjoint copies of $K_{1,r}$ and~$|E(F)| < n_0^2= D$ edges, completing the proof. 
\end{proof}

The second lower bound is inspired by the fact that~$X=X_{n,r,p}$ is approximately Poisson for small~$p$, in which case most~$K_{1,r}$ arise disjointly. 
Indeed, the following standard result bounds~$\Pr(X=m)$ from below by the probability that there are exactly~$m$ vertex-disjoint copies of~$K_{1,r}$ (see~\cite{KkTailDK,Matas2012,APUT} for similar arguments), 
which for~$m = (1+\eps)\mu$ will imply~$\Pr(X \ge m) \ge e^{-O_{r,\eps}(m)}$; the precise form of~\eqref{eq:disjm} will be useful later on. 
\begin{lemma}[Disjoint approximation]\label{lem:disjm}
Given~$r \ge 2$ there are $n_0,b>0$ (depending only on~$r$) such that, for all $n \ge n_0$, $0 < p \le n^{-1-1/(r+1)}$ and integers~$m \in \NN$ satisfying $0 \le m \le 99\max\{\mu,n^{1/(r+1)}\}$, we have 
\begin{equation}\label{eq:disjm}
\Pr(X =m)\ge e^{-b} \cdot \binom{X_{r,n,1}}{m}p^{rm}(1-p^r)^{X_{n,r,1}-m} .
\end{equation}
\end{lemma}
\begin{proof} 
Let $\cK$ contain all copies of $K_{1,r}$ in~$K_n$. Define $\fS_m$ as the collection of all $m$-element subsets of $\cK$ in which all stars~$K_{1,r}$ are vertex disjoint. 
Given~$\cC \subseteq \fS_m$, define $\cI_{\cC}$ as the event that all stars~$K_{1,r}$ of $\cC$ are present, and define $\cD_{\cC}$ as the event that none of the stars~$K_{1,r}$ in $\cK \setminus \cC$ are present. Note that 
\[
\Pr(X=m) \ge \sum_{\cC \in \fS_m} \Pr(\cI_{\cC}\text{ and } \cD_{\cC}) = \sum_{\cC \in \fS_m} \Pr(\cI_{\cC}) \Pr(\cD_{\cC} \mid \cI_{\cC}) \ge |\fS_m| p^{rm} \cdot \min_{\cC \in \fS_m} \Pr(\cD_{\cC} \mid \cI_{\cC}) .
\]
Distinguishing the number of edges in which each star~$\alpha \in \cK \setminus \cC$ overlaps with some star~$K_{1,r}$ from the vertex-disjoint collection~$\cC \in \fS_m$, using Harris inequality~\cite{Harris1960} and~$np =o(1)$ we routinely~obtain 
\[
 \Pr(\cD_{\cC} \mid \cI_{\cC}) \ge (1-p^r)^{X_{n,r,1}-m} \prod_{1 \le j < r}(1-p^{r-j})^{O_r(mn^{r-j})} \ge (1-p^r)^{X_{n,r,1}-m} e^{-O_r(mnp)} ,
\]
where $m np = O(\max\{n^{r+2}p^{r+1},n^{1+1/(r+1)}p\}) = O(1)$. 
Furthermore, with $(z-y)^y/y! \le \binom{z}{y} \le z^y/y!$, $1-x \ge e^{-2x}$ and $X_{n,r,1}=n \binom{n-1}{r}$ in mind, basic counting (and a short calculation) gives  
\[
|\fS_m| \ge \Biggl((n-m) \binom{n-(r+1)m}{r}\Biggr)^m \bigg/m! \ge \binom{X_{n,r,1}}{m} e^{-O_r(m^2/n)}.
\]
This completes the proof of \eqref{eq:disjm} since $m^2 = O(\max\{n^{2(r+1)}p^{2r},n^{2/(r+2)}\}) = O(n^{2/(r+2)})= o(n)$.
\end{proof}

Combining the above two results, we now prove the lower bound of Theorem~\ref{thm:Krp}. 
\begin{proof}[Proof of the lower bound in \eqref{eq:thm:Krp} of Theorem~\ref{thm:Krp}]
We shall tacitly assume~$n \ge n_0(r,\eps)$ whenever necessary. 
Applying Lemma~\ref{lem:cls} with~$x:=(1+\eps)\mu$, there is~$F \subseteq K_n$ with~$|E(F)| \le O_{r,\eps}(\max\{\mu^{1/r},\mu/n^{r-1}\})$ edges that contains at least~$(1+\eps)\mu$ copies of~$K_{1,r}$. 
If~$\Phi = \max\{\mu^{1/r},\mu/n^{r-1}\} \log(1/p)$, then it follows that 
\begin{equation}\label{eq:thm:Krp:cluster}
\Pr(X \ge (1+\eps)\mu) \ge \Pr(F \subseteq G_{n,p})=p^{|E(F)|} 
\ge e^{-O_{r,\eps}(\Phi)}. 
\end{equation}
Otherwise~$\Phi=\mu$, which by a short calculation implies~$\mu \le (\log n)^{3}$, say (since~$\mu \ge (\log n)^{3}$ implies $p = \Omega_r(n^{-1-1/r}) \ge n^{-2}$ and thus $\max\{\mu^{1/r}, \mu/n^{r-1}\}\log(1/p) = O_r(\max \{\mu^{1/r}, \mu/n^{r-1}\} \log n) < \mu$). 
Applying Lemma~\ref{lem:disjm} with~$m:=\ceil{(1+\eps)\mu} < n^{1/(r+1)}$, 
using~$\binom{z}{y} \ge (z/y)^y$, $\mu=X_{n,r,1}p^{r}$, $1-x \ge e^{-2x}$ and~$m \ge (1+\eps)\mu \ge 1$ it follows that  
\begin{equation}\label{eq:thm:Krp:disj}
\Pr(X \ge (1+\eps)\mu) \ge \Pr(X=m) \ge e^{-O_r(1)} \cdot \left(\frac{\mu}{m}\right)^m e^{-2\mu} \ge e^{-\Theta_{r,\eps}(m)} \ge e^{-O_{r,\eps}(\Phi)} ,
\end{equation}
establishing the lower bound in~\eqref{eq:thm:Krp}. 
\end{proof}

\subsection{Refined arguments for Theorem~\ref{thm:Kri} and~\ref{thm:Kripc}}\label{sec:coreext:LB}
For Theorem~\ref{thm:Kri} and~\ref{thm:Kripc} we shall refine the previous two lower bounds, and also introduce a new third lower bound. Each time some care is needed to obtain the `correct' dependence on~$t=\eps \mu$ in the exponent, 
and we start by refining the `clustering' based lower bound from Lemma~\ref{lem:cls} and~\eqref{eq:thm:Krp:cluster}. 
\begin{lemma}[Refined clustering bound]\label{lem:cls:refined}
Given~$r \ge 1$ and~$\xi \in (0,1)$ there are~$n_0,c>0$ (depending only on~$r,\xi$) such that, for all $n \ge n_0$, $p \in (0,1-\xi]$ and $t \ge \sigma$ satisfying~$1 \le \mu + t \le X_{r,n,1}$, we have 
\begin{equation}\label{eq:cls:refined}
\Pr(X \ge \mu+t) \ge \exp\left(-c \max\{t^{1/r},t/n^{r-1}\}\log(1/p)\right) .
\end{equation}
\end{lemma}
\noindent
In case of~$p=o(1)$ the basic proof idea is to obtain~$\mu+t$ copies of~$K_{1,r}$ as follows: (i)~we first use the clustering construction from Lemma~\ref{lem:cls} to plant~$2t$ copies of~$K_{1,r}$, and (ii)~then use Harris' inequality and a one-sided Chebychev's inequality to show that typically~$\ge \mu- t$ of the \emph{remaining} $\tilde{X}_{n,r,1}:=X_{n,r,1}-2t$ other copies of~$K_{1,r}$ are present (the crux is that the \emph{expected} number of such copies is~$\tilde{X}_{n,r,1} p^{r} = \mu-o(t)$, so having~$\ge \mu -t$ of them intuitively seems likely). 
For the resulting lower bound step~(i) with probability~$p^{O_r(\max\{t^{1/r},t/n^{r-1}\})}$ thus ought to give the main contribution, making~\eqref{eq:cls:refined} plausible. 
For technical reasons, in the actual argument we have to plant~$\min\{(\beta+1)t,\ceil{\mu+t}\}$ copies of~$K_{1,r}$ for carefully chosen~$\beta>0$. 
By mimicking the proof of Theorem~21 in~\cite{APUT} we then easily arrive at~\eqref{eq:cls:refined} above; we leave the details to the reader. 

We next refine the `disjoint approximation' based lower bound used in Lemma~\ref{lem:disjm} and~\eqref{eq:thm:Krp:disj} for small~$p$.  
The idea is that inequality~\eqref{eq:disjm} intuitively relates~$X=X_{n,r,p}$ to a binomial random variable with mean~$\mu=X_{n,r,1}\cdot p^r$, 
which makes the following Chernoff-type bound for the upper tail plausible. 
\begin{lemma}[Disjoint approximation: Chernoff-type lower bound]\label{lem:disj} 
Given~$r \ge 2$ there are~$n_0,c,d>0$ (depending only on~$r$) such that,  
for all $n \ge n_0$, $0 < p \le n^{-1-1/(r+1)}$ and~$t > 0$ satisfying~$1 \le \mu + t \le 9\max\{\mu,n^{1/(r+1)}\}$, we have 
\begin{equation}\label{eq:disj}
\Pr(X \ge \mu + t) \ge d \exp\left(-c\varphi(t/\mu) \mu  \right) . 
\end{equation}
\end{lemma}
%
\noindent
Noting the binomial-like form of inequality~\eqref{eq:disjm} it is routine to check that Lemma~\ref{lem:disjm} indeed implies~\eqref{eq:disj} above 
(e.g., by summing~\eqref{eq:disjm} as in the proof of Theorem~22 in~\cite{APUT}); 
we leave the details to the reader. 

Our third lower bound for moderately large~$p$ 
it is based on the idea that a deviation in the number of edges should typically entail a deviation in the number of~$K_{1,r}$ copies 
(in concrete words: if~$G_{n,p}$ has substantially more than~$\binom{n}{2}p$ edges, then we expect to have more~$K_{1,r}$ copies than on~average). 
\begin{lemma}[Deviation in number of edges: sub-Gaussian type lower bound]\label{lem:edges}
Given~$r \ge 2$ and~$\xi \in (0,1)$ there are~$n_0,\beta,c>0$ (depending only on~$r,\xi$) such that, 
setting $\Lambda := \mu(1+(np)^{r-1})$, 
for all~$n \ge n_0$, $\xi n^{-1} \le p \le 1-\xi$ and~$\sigma \le t \le \beta\mu$ we have
\begin{equation}\label{eq:edges}
\Pr(X \ge \mu+t)\ge \exp\bigl(-c\varphi(t/\mu) \mu^2/\Lambda\bigr) \ge \exp\bigl(-ct^2/\Lambda\bigr) .
\end{equation}
\end{lemma}
\begin{remark}\label{rem:edges}
By Remark~\ref{rem:variance}, in inequality~\eqref{eq:edges} we have~$\Lambda = \Theta_{r,\xi}(\sigma^2)$, where $\sigma^2 = \Var X$.
\end{remark}
\noindent 
Setting~$\eps:=t/\mu$, the basic proof idea is to (i)~condition on having~$|E(G_{n,p})| \ge (1+\eps)\binom{n}{2}p$ edges, 
and (ii)~then show that this conditioning converts~$X \ge \mu+t=(1+\eps)\mu$ into a typical event 
(the crux is that this conditioning drives up the expected value of $X=X_{n,r,p}$; to see this it might help to think of the uniform random graph~$G_{n,m}$ with~$m=(1+\eps)\binom{n}{2}p$ edges). 
For the resulting lower bound the conditioning thus ought to give the main contribution, which by folklore results satisfies $\Pr(|E(G_{n,p})| \ge (1+\eps)\binom{n}{2}p) = \exp\bigl(-\Theta_\xi(\eps^2 \binom{n}{2}p))\bigr)$. 
This makes inequality~\eqref{eq:edges} plausible, since~$\eps^2\binom{n}{2}p = \eps^2 \cdot \Theta_{r, \xi}(\mu^2/\Lambda) = \Theta_{r, \xi}(t^2/\Lambda)$ for the considered range of~$p$. 
A simple modification of the proof of Theorem~24 in~\cite{APUT} makes this idea rigorous and establishes~\eqref{eq:edges} above; 
we leave the details to the reader (we mention in passing that a tilting argument also works~here). 

Stitching the above three results together, we now prove the lower bounds of Theorem~\ref{thm:Kri}~and~\ref{thm:Kripc}.  
\begin{proof}[Proof of the lower bound in~\eqref{eq:thm:Kripc} of Theorem~\ref{thm:Kripc}]
By~(iii) of Lemma~\ref{cl:phi_to_t_square} we infer that~$M \ge t^{1/r} =\Omega_{r,\xi}(1)$, which in turn implies $t^2/\sigma^2 \ge M \log(e/p) \cdot (\log n)^{2r} \ge 1$ and thus~$t \ge \sigma$. 
Hence an application of Lemma~\ref{lem:cls:refined} (see inequality~\eqref{eq:cls:refined}) establishes the lower bound in~\eqref{eq:thm:Kripc}. 
\end{proof}
\begin{proof}[Proof of the lower bound in~\eqref{eq:thm:Krip} of Theorem~\ref{thm:Kripc}]
We shall only assume~$p \ge n^{-1}$ instead of~$p \ge n^{-1}\log n$. 
Applying Lemmas~\ref{lem:cls:refined} and~\ref{lem:edges}, and using Remark~\ref{rem:edges}, it follows that there is~$\beta=\beta(r,\xi)>0$ such that 
\begin{equation}\label{eq:thm:Krip:lower:indic}
\Pr(X \ge \mu+t) \ge \max\left\{e^{ -\Theta_{r,\xi}\left( M \log(e/p) \right) }, \; \indic{t \le \beta \mu}e^{ -\Theta_{r,\xi}\left(t^2/\sigma^2\right) }\right\}.
\end{equation}
By a virtually identical calculation as in the proof of~\eqref{eq:no_normal_term} from Lemma~\ref{cl:phi_to_t_square}, 
for~$t \ge \beta \mu$ it follows that~$t^2/\sigma^2 \ge \Omega_{r,\xi}(M \log(e/p))$ holds. 
After adjusting the implicit constants, it follows that we can remove the indicator in inequality~\eqref{eq:thm:Krip:lower:indic}, 
which in view of~$\Psi(t) = \min\{t^2/\sigma^2, M\log (e/p)\}$ establishes the lower bound in~\eqref{eq:thm:Krip}. 
\end{proof}
\begin{proof}[Proof of the lower bound in~\eqref{eq:thm:Kri} of Theorem~\ref{thm:Kri}]
Set~$t := \eps \mu$ and~$M := \max\{ t^{1/r}, t/n^{r-1}\}$, as usual.
Using~\eqref{eq:varphi} we have $(\eps \mu)^2/\sigma^2 \ge \varphi(\eps)\mu^2/\sigma^2 \ge \Phi(\eps) \ge 1$ by assumption, so~$t \ge \sigma$ follows. 
In the following we shall distinguish the three cases (i)~$n^{-1} \le p \le 1 - \xi$, (ii)~$n^{-1-1/(r+1)} \le p < n^{-1}$, and (iii)~$0 < p < n^{-1-1/(1+r)}$. 

In cases~(i)--(ii) note that, say,~$\mu^{1-1/r} =\Omega_r(n^{1/3r}) > \log n$ holds. 
Using~(i)--(ii) of Lemma~\ref{cl:phi_to_t_square}, it thus suffices to prove the lower bound of~\eqref{eq:thm:Kri} with exponent~$\Phi(\eps)$ replaced by~$\Psi(t)$ defined in \eqref{eq:thm:Krip}.  
In case~(i) this bound follows from the above proof (valid for~$n^{-1} \le p \le 1 - \xi$) of the lower bound in~\eqref{eq:thm:Krip}, 
and in case~(ii) we shall now argue that this bound follows from inequality~\eqref{eq:cls:refined} of Lemma~\ref{lem:cls:refined}, 
by establishing that~$t^2/\sigma^2 = \Omega_{r,\xi}(M \log(e/p))$ holds. 
Indeed, since~$p < n^{-1}$ and Remark~\ref{rem:variance} imply~$\sigma^{2} = \Theta_{r}(\mu)$, 
after recalling $\mu^{1-1/r} =\Omega_r(n^{1/3r})$ and~$t =\eps \mu \ge n^{-\alpha}\mu$ it then follows for~$\alpha=\alpha(r)>0$ sufficiently small~(say,~$\alpha < 1/6r$)~that 
\begin{equation}\label{eq:thm:Krip:lower:indic:calc}
	\frac{t^2/\sigma^2}{M \log(e/p)} = \frac{\min\left\{{t^{2-1/r}}, \; {t n^{r-1}}\right\}}{\sigma^2\log(e/p)} \ge \frac{\min\left\{\mu^{1-1/r}n^{-2\alpha}, \: n^{r-1-\alpha} \right\}}{\Theta_{r,\xi}(\log n)} \ge 1 ,
\end{equation}
completing the proof in cases~(i)--(ii). 

In the remaining case~(iii) Lemmas \ref{lem:cls:refined} and \ref{lem:disj} imply that, for some constant~$d = d(r) \in (0,1]$, we have 
\begin{equation}
	\label{eq:case3}
	\Pr(X \ge \mu + t) \ge d \cdot \max \left\{ e^{ -\Theta_{r,\xi}\left( M \log(e/p) \right) }, \; \indic{\mu + t \le 9\max \{\mu, n^{1/(r+1)}\}} e^{- \Theta_{r,\xi}\left( \varphi(t/\mu)\mu \right) } \right\}.
\end{equation}
We claim that for $\mu + t > 9 \max \left\{  \mu, n^{1/(r+1)}\right\}$ we have $\varphi(t/\mu)\mu = \Omega_{r}\left( M \log(e/p) \right)$. 
Indeed, noting that~$\varphi(x) \ge x(\log x)/2$ for~$x \ge e^2 \approx 7.4$ 
(which is easy to check by calculus), it follows that 
\begin{equation*}
\frac{\varphi(t/\mu) \mu}{M \log(e/p)} \ge \frac{\min\{t^{(r-1)/r}, \: n^{r-1}\} \log(t/\mu)}{2\log(e/p)} = \Omega_r\left(n^{(r-1)/(r^2+r)} \cdot \frac{\log(t/\mu)}{\log(e/p)}\right) .
\end{equation*} 
Furthermore, $\log(t/\mu)/\log(e/p) = \Omega_r(1)$ when~$\mu \le p$, and~$\log(t/\mu)/\log(e/p) = \Omega_r((\log n)^{-1})$ when~$\mu > p$. 
In each case the claimed inequality holds, 
which allows omitting the indicator in~\eqref{eq:case3}. 
Since~$\mu = \Theta_{r}(\mu^2/\sigma^2)$ by Remark~\ref{rem:variance}, 
now~$\Pr(X \ge \mu + t) \ge d \cdot e^{-\Theta_{r,\xi}(\Phi(\eps))}$ follows, 
which in view of~$\Phi(\eps) \ge 1$ completes the~proof. 
\end{proof}

\end{appendix}

\end{document}